\documentclass[12pt]{article}
\usepackage{amsmath,amsfonts,amsthm,amssymb, mathtools}
\usepackage{mathrsfs,graphicx,color,latexsym, tikz,calc,enumerate,indentfirst}
\usetikzlibrary{shadows}
\usepackage{tikz-cd}
\usetikzlibrary{patterns,arrows,decorations.pathreplacing}
\usepackage[colorlinks=true,
linkcolor=blue,citecolor=blue,
urlcolor=blue]{hyperref}

\newtheorem{theorem}{Theorem}[section]
\newtheorem{lemma}[theorem]{Lemma}

\newtheorem{corollary}[theorem]{Corollary}
\newtheorem{remark}[theorem]{Remark}

\renewcommand\proofname{\bf{Proof}}

\allowdisplaybreaks

\title{\bf A note on the Bj\"{o}rner--Kalai theorem}

\author{Xiongfeng Zhan,  Xueyi Huang\thanks{Corresponding author.} \setcounter{footnote}{-1}\footnote{\emph{Email address:} zhanxfmath@163.com (X. Zhan), huangxy@ecust.edu.cn (X. Huang).}\\[2mm]
	\small School of Mathematics, East China University of Science and Technology, \\
	\small  Shanghai 200237, P. R. China
}

\date{}

\begin{document}
	\maketitle
	\begin{abstract}
 
 In 1988, Bj\"{o}rner and Kalai used combinatorial shadow functions to characterize the maximal Betti sequence for a given $f$-vector and the minimal $f$-vector for a given Betti sequence. Their description of the maximal Betti sequence was expressed through a set of inequalities. In this paper, we introduce an error function $\delta_k$ associated with the combinatorial shadow functions and use it to sharpen these inequalities into exact equalities. As a corollary, we obtain an equivalent form of Bj\"{o}rner and Kalai's characterization of all possible pairs $(f,\beta)$ that can occur as the $f$-vector and Betti sequence of a simplicial complex. Moreover, combining our results with a previous result of Bj\"{o}rner in 2011, we derive a new number-theoretic inequality concerning the count of odd square-free integers with a specified number of prime factors.

		\par\vspace{2mm}
		
		\noindent{\bfseries Keywords:} combinatorial shadow function, simplicial complex, $f$-vector, Betti number
		\par\vspace{1mm}
		
		\noindent{\bfseries 2010 MSC:} 05E45
	\end{abstract}

	\section{Introduction}\label{section::1}
	
	Let $k\geq 1$ be a fixed integer. According to number theory, for any integer  $n\geq 1$, there exists a unique sequence of integers $a_k>a_{k-1}>\cdots>a_i\geq i\geq 1$ such that
	\begin{equation}\label{equ::num}
		n=\binom{a_{k}}{k}+\binom{a_{k-1}}{k-1}+\cdots+\binom{a_i}{i}.
	\end{equation}
	This unique expansion allows one to define the following \textit{combinatorial shadow functions}: 
	\begin{align*}
		\partial_{k-1}(n)&=\binom{a_{k}}{k-1}+\binom{a_{k-1}}{k-2}+\cdots+\binom{a_i}{i-1},\\
		\partial^{k-1}(n)&=\binom{a_{k}-1}{k}+\binom{a_{k-1}-1}{k-1}+\cdots+\binom{a_i-1}{i}.
	\end{align*}
	For completeness, let $\partial_{k-1}(0)=\partial^{k-1}(0)=0$.
	
	Let $\mathbb{N}_0^{(\infty)}$ denote the set of ultimately vanishing sequences of nonnegative integers. The set $\mathbb{N}_0^{(\infty)}$ and all its subsets are naturally equipped with the \emph{componentwise partial order}, defined by
	\[
	(n_0, n_1, \ldots) \leq (n_0', n_1', \ldots) \quad \Longleftrightarrow \quad n_i \leq n_i' \text{ for all } i \geq 0.
	\]
	
	For an ordered pair $(f = (f_0, f_1, \ldots), \beta = (\beta_0, \beta_1, \beta_2, \ldots))$ of sequences from $\mathbb{N}_0^{(\infty)}$, we define  
	\[
	\chi_{k-1} := \chi_{k-1}(f, \beta) = \sum_{j \geq k} (-1)^{j-k} (f_j - \beta_j), \quad \text{for each } k \geq 0.
	\]
	Clearly, $f_k - \chi_{k-1} = \chi_{k} + \beta_k$.
	
	Let $\Delta$ be a finite simplicial complex. The \textit{$f$-vector} of $\Delta$ is the sequence $f = (f_0, f_1, \ldots)$, where 
	\[ f_i = \#\{F \in \Delta \mid \dim F = i\}, \quad \text{for each } i \geq 0. \]
	
	Let $\boldsymbol{k}$ be a fixed field, and let $\widetilde{H}_i(\Delta; \boldsymbol{k})$ denote the $i$-th reduced simplicial homology group of $\Delta$ over $\boldsymbol{k}$. The \textit{reduced Betti sequence} of $\Delta$ over $\boldsymbol{k}$ is defined as $\beta = (\beta_0, \beta_1, \ldots)$, where 
	\[ \beta_i = \dim_{\boldsymbol{k}} \widetilde{H}_i(\Delta; \boldsymbol{k}), \quad \text{for each } i \geq 0. \]
	
	Using the combinatorial shadow function $\partial_k$, Bj\"{o}rner and Kalai \cite{BK88} provided a complete characterization of all possible pairs $(f, \beta)$ that may occur as the $f$-vector and Betti sequence of a finite simplicial complex, extending the well-known Euler--Poincar\'{e} theorem (cf. \cite{Poi93,Poi99}):
\[
\chi_{-1} = \sum_{j \geq 0} (-1)^{j} (f_j - \beta_j) = 1.
\]

		\begin{theorem}[{Bj\"{o}rner and Kalai, \cite[Theorem 1.1]{BK88}}]\label{thm::1}
		Suppose that $f=(f_0,f_1,\ldots),\beta=(\beta_0,\beta_1,\ldots)\in \mathbb{N}_0^{(\infty)}$ are two given  sequences and $\boldsymbol{k}$ is a field. Then the following conditions are equivalent:
		\begin{enumerate}[$(a)$]\setlength{\itemsep}{0pt}
			\item $f$ is the $f$-vector and $\beta$ the Betti sequence over $\boldsymbol{k}$ of some simplicial complex,
			\item  $\chi_{-1}=1$, and $\partial_k(\chi_{k}+\beta_k)\leq \chi_{k-1}$ for all $k\geq 1$.
		\end{enumerate}
	\end{theorem}

	An ordered pair $(f, \beta)$ of sequences from $\mathbb{N}_0^{(\infty)}$ is called \textit{compatible}, denoted by $f \sim \beta$, if there exists a simplicial complex with $f$-vector $f$ and Betti sequence $\beta$. By Theorem \ref{thm::1}, this relation is purely combinatorial and independent of the field characteristic.

	In their seminal work \cite{BK88}, Bj\"{o}rner and Kalai further established the following foundational result concerning compatible pairs  $(f, \beta)$.

	\begin{theorem}[{Bj\"{o}rner and Kalai, \cite[Theorem 1.2, Theorem 5.3]{BK88}}]\label{thm::2}
		\begin{enumerate}[$(a)$]\setlength{\itemsep}{0pt}
			\item Suppose that $f$ is the $f$-vector of some simplicial complex. Then the set $B_{f}=\{\beta\in \mathbb{N}_0^{(\infty)}: f\sim\beta\}$ has a unique maximal element. Define $\psi(f)=\max B_f$.
			\item Suppose that  $\beta\in \mathbb{N}_0^{(\infty)}$. The set $F_{\beta}=\{f\in \mathbb{N}_0^{(\infty)}: f\sim\beta\}$ has a unique minimal element. Define $\phi(\beta)=\min F_\beta$.
			\item Suppose that $f$ is the $f$-vector of some simplicial complex, and $\beta\in \mathbb{N}_0^{(\infty)}$.  Then $\beta=\psi(f)$ if and only if $\chi_{-1}=1$, and $\partial_k(\chi_{k}+\beta_k)\leq\chi_{k-1}\leq\partial_k(\chi_{k}+\beta_k+1)$ for all $k\geq 1$.
			\item Suppose that $(f,\beta)$ is a compatible pair. Then $f=\phi(\beta)$ if and only if  $\partial_k(\chi_{k}+\beta_k)=\chi_{k-1}$ for all $k\geq 1$.
			\item $\psi(\phi(\beta))=\beta$ for all $\beta\in\mathbb{N}_0^{(\infty)}$.
			\item $\phi(\psi(f))\leq f$ for all $f$-vectors $f$. 
		\end{enumerate}
	\end{theorem}
	
	We note that only the maximal part contains inequalities, which inspired us to introduce a new function to measure the discrepancy in these inequalities. The \textit{error function} $\delta_{k-1}$ is defined by 
\[
\delta_{k-1}(n)=\ell,
\]
where $\ell$ is the number of indices $t$ in the unique expansion \eqref{equ::num} for which $a_t = t$. For completeness, we define $\delta_{k-1}(0)=0$.
	
	In this paper, we prove the following result, which shows how the error function $\delta_{k}$ affects inequalities related to the combinatorial shadow functions $\partial_k$ and $\partial^k$.
	
	\begin{theorem}\label{thm::main1}
		Let $n$, $m$  be nonnegative integers and let  $\epsilon_k$ be an integer with  $0\leq\epsilon_k\leq \delta_k(n+m)$. Then the following statements hold:
		\begin{enumerate}[$(a)$]\setlength{\itemsep}{0pt}
			\item $\partial_k(n)\leq m-\epsilon_k$ if and only if $ n\leq \partial^k(n+m)$;
			\item $n=\partial^k(n+m)$ if and only if $\partial_k(n)+\delta_k(n+m)=m$;
			\item $\partial_k(n)=m$ if and only if $n= \partial^k(n+m)$ and $\delta_k(n+m)=0$.
		\end{enumerate}
	\end{theorem}

	As an application of Theorem \ref{thm::main1}, we sharpen the inequalities in parts $(c)$ and $(f)$ of Theorem \ref{thm::2} and provide an equivalent reformulation of part $(d)$.  For a sequence $f = (f_0, f_1, \ldots) \in \mathbb{N}_0^{(\infty)}$,  we define
	\[
	\delta(f) = (0, \delta_1(f_1), \delta_2(f_2), \ldots) \quad \text{and} \quad \delta_+(f) = (\delta_1(f_1), \delta_2(f_2), \ldots).
	\]

	\begin{theorem}\label{thm::main2}
		\begin{enumerate}
			\item[$(a)$] Suppose that $f$ is the $f$-vector of some simplicial complex. Then the set $B_{f}=\{\beta\in \mathbb{N}_0^{(\infty)}: f\sim\beta\}$ has a unique maximal element. Define $\psi(f)=\max B_f$.
			\item[$(b)$] Suppose that  $\beta\in \mathbb{N}_0^{(\infty)}$. The set $F_{\beta}=\{f\in \mathbb{N}_0^{(\infty)}: f\sim\beta\}$ has a unique minimal element. Define $\phi(\beta)=\min F_\beta$.
			\item[$(c^*)$] Suppose that $f$ is the $f$-vector of some simplicial complex, and $\beta\in \mathbb{N}_0^{(\infty)}$.  Then $\beta=\psi(f)$ if and only if  $\chi_{-1}=1$, and $\partial_k(\chi_{k}+\beta_k)+\delta_k(f_k)=\chi_{k-1}$ for all $k\geq 1$, which is the case if and only if $\chi_{-1}=1$, and $f_k-\chi_{k-1}=\partial^k(f_k)$ for all $k\geq 1$.
			\item[$(d^*)$] Suppose that $(f,\beta)$ is a compatible pair. Then $f=\phi(\beta)$ if and only if  $\partial_k(\chi_{k}+\beta_k)=\chi_{k-1}$ for all $k\geq 1$, which is the case if and only if $ f_k-\chi_{k-1}=\partial^k(f_k)$ for all $k\geq 1$, and $\delta(f)=0$.
			\item[$(e)$] $\psi(\phi(\beta))=\beta$ for all $\beta\in\mathbb{N}_0^{(\infty)}$.
			\item[$(f^*)$] $\phi(\psi(f)+\delta_+(f))=f-\delta(f)$ for all $f$-vectors $f$. Therefore,  $\phi(\psi(f))\leq f-\delta(f)-\delta_+(f)$, with equality holding  if and only if $\delta(f)=0$.
		\end{enumerate}
	\end{theorem}
	
By applying Theorem \ref{thm::main1} to Theorem \ref{thm::1}, we also obtain the following result.

	\begin{corollary}\label{cor::1}
		Suppose that $f=(f_0,f_1,\ldots),\beta=(\beta_0,\beta_1,\ldots)\in \mathbb{N}_0^{(\infty)}$ are two given  sequences and $\boldsymbol{k}$ is a field. Let $\epsilon\in \mathbb{N}_0^{(\infty)}$ be any sequence satisfying $0\leq\epsilon\leq \delta(f)$. Then the following conditions are equivalent:
		\begin{enumerate}\setlength{\itemsep}{0pt}
			\item[$(a)$] $f$ is the $f$-vector and $\beta$ the Betti sequence over $\boldsymbol{k}$ of some simplicial complex,
			\item[$(b^*)$]   $\chi_{-1}=1$, and $\partial_k(\chi_{k}+\beta_k)\leq \chi_{k-1}-\epsilon_k$ for all $k\geq 1$,
			\item[$(b^+)$]  $\chi_{-1}=1$, and $0\leq f_k-\chi_{k-1}\leq \partial^k(f_k)$ for all $k\geq 1$.
		\end{enumerate}
	\end{corollary}
	\begin{remark}	\rm
		It should be noted that  condition  $(b^*)$ in Corollary \ref{cor::1} provides a refinement of condition $(b)$ from Theorem \ref{thm::1}, while condition $(b^+)$ establishes a new characterization of compatible pairs $(f,\beta)$. Furthermore, for any simplicial complex,  the equality 
		\[
		f_k - \chi_{k-1} = \dim Z_k
		\] 
		holds, where $Z_k$ represents the space of $k$-cycles (see \cite[Remark 4.4]{BK88}). Therefore, through homological interpretation of the inequality   $f_k - \chi_{k-1} \leq \partial^k(f_k)$ in Corollary \ref{cor::1} $(b^+)$, we obtain 
		\[
		\dim Z_k \leq \partial^k(f_k).
		\]  
	\end{remark}

Let $m$ be a square-free positive integer, and let $P(m)$ be the set of prime factors of $m$. The family  
\[
\Delta_n := \{P(m): m \text{ is square-free and } m \leq n\}
\]
is a simplicial complex, called the \textit{Bj\"{o}rner complex}. It was shown by Bj\"{o}rner \cite{B11} that the Prime Number Theorem and the Riemann Hypothesis are equivalent to certain asymptotic estimates of the reduced Euler characteristics of these complexes as $n \to \infty$. 

Let $\sigma_{k}(n)$ (resp. $\sigma_{k}^{\mathrm{odd}}(n)$) denote the number of square-free integers (resp. odd square-free integers) $\leq n$ with exactly $k$ prime factors. In \cite{B11}, Bj\"{o}rner applied Theorem \ref{thm::1} to the Bj\"{o}rner complex and proved that
\[
\partial_k(\sigma_{k+1}^{\mathrm{odd}}(n))\leq\sigma_{k}^{\mathrm{odd}}(n/2)\quad \text{for all } k\geq 1.
\] 
Combining this result with Theorem  \ref{thm::main1}, we immediately derive the following number-theoretic inequality.
	\begin{corollary}
		For all $k\geq1$,
		\begin{equation*}\label{eq::sigma_k}
			\sigma_{k+1}^{\mathrm{odd}}(n)\leq\partial^k(\sigma_{k+1}(n)).
		\end{equation*}
	\end{corollary}


	\section{Proofs of Theorems \ref{thm::main1} and \ref{thm::main2}}\label{section::2}

	In order to prove Theorem \ref{thm::main1}, we need a series of number-theoretic lemmas.
	
	\begin{lemma}[{\cite{BK88}}, {\cite[Lemma 3.1]{Fro08}}]\label{lem::2.1}
		Let $n$ and $k$ be positive integers. Then there is a unique expansion 
		\begin{equation*}
			n=\binom{a_{k}}{k}+\binom{a_{k-1}}{k-1}+\cdots+\binom{a_i}{i},
		\end{equation*}
		such that $a_k>a_{k-1}>\cdots>a_i\geq i\geq 1$. This expansion is called the $(k-1)$-dimensional representation of $n$.
	\end{lemma}
	
	\begin{lemma}\label{lem::2.2}
		Let $n$, $m$, and $k$ be positive integers. Suppose that  
		\[
		n = \binom{a_{k+1}}{k+1} + \binom{a_k}{k} + \cdots + \binom{a_i}{i}
		\]
		and
		\[
		m = \binom{b_{k+1}}{k+1} + \binom{b_k}{k} + \cdots + \binom{b_j}{j}
		\]
		are the $k$-dimensional representations of $n$ and $m$, respectively. Then the following statements hold:
		\begin{enumerate}[$(a)$]
			\item $n=m$ if and only if $i=j$ and $a_t = b_t$ for all $t\in\{i,\ldots,k+1\}$;
			\item $n < m$ if and only if either
			$i > j$ and $a_t = b_t$ for all $t\in\{i,\ldots,k+1\}$, or the maximal index $t\in\{\max\{i,j\},\ldots,k+1\}$ with  $a_t \neq b_t$ satisfies $a_t < b_t$.
		\end{enumerate}
	\end{lemma}
	
	\begin{proof}
		Note that $(a)$ follows from Lemma \ref{lem::2.1} immediately. Thus it suffices to consider $(b)$.   Using Pascal's identity $\binom{p}{q} = \binom{p-1}{q} + \binom{p-1}{q-1}$ recursively, we obtain
		\[
		\binom{c+1}{d} = \binom{c}{d} + \binom{c-1}{d-1} + \cdots + \binom{c-d}{0}
		\]
		for any integers $0 \leq d \leq c+1$. This implies that
		\begin{equation}\label{equ::com}
			\binom{c_p}{p} > \binom{d_p}{p} + \binom{d_{p-1}}{p-1} + \cdots + \binom{d_l}{l}
		\end{equation}
		whenever $c_p > d_p$ and $d_p > d_{p-1} > \cdots > d_l \geq l \geq 1$. Let $s=\max\{i,j\}$. If $a_t = b_t$ for all $t\in \{s,\ldots,k+1\}$, then $n<m$ when $i>j$, $n>m$ when $i<j$, and $n=m$ when $i=j$. Now suppose that  there exists some index $t\in\{s,\ldots,k+1\}$ such that $a_t \neq b_t$. Let $t_0=\max\{t\in \{s,\ldots,k+1\}: a_t\neq b_t\}$. Then it follows  from \eqref{equ::com} that $n<m$ when $a_{t_0}<b_{t_0}$, and $n>m$ when  $a_{t_0}>b_{t_0}$.  This proves $(b)$.
	\end{proof}
	
	For any integer  $n,k\geq 1$, using the unique expansion \eqref{equ::num}, we define: 
	\[d^{k-1}(n)=\binom{a_{k}+1}{k}+\binom{a_{k-1}+1}{k-1}+\cdots+\binom{a_i+1}{i}.\]
	For completeness, define $d^{k-1}(0)=0$.

	\begin{lemma}\label{lem::2.3}
		Let $n$ be a nonnegative integer. Then $n+\partial_k(n)=d^k(n)$.
	\end{lemma}
	
	\begin{proof}
		The case $n = 0$ is trivial, so we may assume $n \geq 1$. By Lemma \ref{lem::2.1}, we suppose that the $k$-dimensional representation of $n$ is given by
		\[
		n = \binom{a_{k+1}}{k+1} + \binom{a_k}{k} + \cdots + \binom{a_i}{i},
		\]
		where $a_{k+1} > a_k > \cdots > a_i \geq i \geq 1$. Thus we have
		\begin{equation*}
			\begin{aligned}
				n+\partial_k(n)&=\binom{a_{k+1}}{k+1}+\binom{a_k}{k}+\cdots+\binom{a_i}{i}+\binom{a_{k+1}}{k}+\binom{a_k}{k-1}+\cdots+\binom{a_i}{i-1}\\
				&=\binom{a_{k+1}+1}{k+1}+\binom{a_k+1}{k}+\cdots+\binom{a_i+1}{i}\\
				&=d^k(n),
			\end{aligned}
		\end{equation*}
		as desired.
	\end{proof}

	\begin{lemma}\label{lem::2.4}
		Let $n$ be a nonnegative integer. Then $\delta_k(n)=n-\partial^k(n)-\partial_k(\partial^k(n))$.
	\end{lemma}
	
	\begin{proof}
		The case $n = 0$ is trivial, so we may assume $n \geq 1$. By Lemma \ref{lem::2.1},  we suppose that the $k$-dimensional representation of $n$ is given by
		\[
		n=\binom{a_{k+1}}{k+1}+\binom{a_k}{k}+\cdots+\binom{a_i}{i},
		\] 
		where $a_{k+1}>a_k>\cdots>a_i\geq i\geq 1$. If $a_t=t$ for all $t\in\{i,\ldots, k+1\}$, then $\delta_k(n)=n=k+2-i$, $\partial^k(n)=\partial_k(\partial^k(n))=0$, and we are done. Now suppose that there exists some $t\in\{i,\ldots, k+1\}$ such that $a_t\neq t$. Let $t_0=\min\{t\in\{i,\ldots, k+1\}: a_t\neq t\}$. Then  $a_{t}=t$ when $t<t_0$, and $a_t>t$ when $t\geq t_0$. According to the definition of the operator $\delta_k$, we obtain
		\[
		\delta_k(n)=t_0-i=\binom{a_{t_0-1}}{t_0-1}+\cdots+\binom{a_i}{i}.
		\]
		Therefore, by Lemma \ref{lem::2.3}, 
		\[
		\begin{aligned}
			\partial^k(n)+\partial_k(\partial^k(n))&=d^k(\partial^k(n))\\
			&=d^k\left(\binom{a_{k+1}-1}{k+1}+\binom{a_k-1}{k}+\cdots+\binom{a_{i_0}-1}{i_0}\right)\\
			&=\binom{a_{k+1}}{k+1}+\binom{a_k}{k}+\cdots+\binom{a_{i_0}}{i_0}\\
			&=n-\delta(n),
		\end{aligned}
		\]
		as desired.
	\end{proof}
	\begin{lemma}\label{lem::2.5}
		Let $n$ and $m$ be nonnegative integers. 
		\begin{enumerate}[$(a)$]\setlength{\itemsep}{0pt}
			\item If $d^k(n)\leq m$, then $n\leq \partial^k(m)$. In particular, if $d^k(n)=m$, then $n= \partial^k(m)$.
			\item If $n\leq m$, then $\partial_k(n)\leq \partial_k(m)$. 
			\item If $d^k(n)=d^k(m)$, then $n=m$.
		\end{enumerate}
	\end{lemma}
	\begin{proof}
		If $n=0$ or $m=0$, the result is trivial. Thus we may assume  that $n,m\geq1$. By Lemma \ref{lem::2.1}, we suppose that the $k$-dimensional representations of $n$ and $m$ are  given by
		\[
		n=\binom{a_{k+1}}{k+1}+\binom{a_k}{k}+\cdots+\binom{a_i}{i}\]
		and 
		\[
		m=\binom{b_{k+1}}{k+1}+\binom{b_k}{k}+\cdots+\binom{b_j}{j},
		\]
		where $a_{k+1}>a_k>\cdots> a_i\geq i\geq 1$ and  $b_{k+1}>b_k>\cdots> b_j\geq j\geq 1$, respectively. Then 
		\[
		d^k(n)=\binom{a_{k+1}+1}{k+1}+\binom{a_k+1}{k}+\cdots+\binom{a_i+1}{i}
		\]
		and 
		\[
		d^k(m)=\binom{b_{k+1}+1}{k+1}+\binom{b_k+1}{k}+\cdots+\binom{b_j+1}{j}
		\]
		are  the $k$-dimensional representations of $d^k(n)$ and $d^k(m)$, respectively. 
		
		First consider $(a)$. If $d^k(n)=m$, by Lemma \ref{lem::2.2}, we obtain $i=j$ and $a_t+1=b_t$ for all $t\in\{i,\ldots,k+1\}$. Therefore,  $n=\partial^k(m)$. If $d^k(n)< m$, again by Lemma \ref{lem::2.2}, either $i>j$ and $a_t+1=b_t$ for all $t\in\{i,\ldots,k+1\}$, or the maximal index $t\in\{\max\{i,j\},\ldots,k+1\}$ with $a_t+1\neq b_t$ satisfies $a_t+1<b_t$. For the former case, we have 
		\begin{equation*}
			\begin{aligned}
				n&=\binom{a_{k+1}}{k+1}+\binom{a_k}{k}+\cdots+\binom{a_i}{i}\\
				&= \binom{b_{k+1}-1}{k+1}+\binom{b_k-1}{k}+\cdots+\binom{b_i-1}{i}\\
				&\leq \binom{b_{k+1}-1}{k+1}+\binom{b_k-1}{k}+\cdots+\binom{b_i-1}{i}+\cdots+\binom{b_j-1}{j}\\
				&= \partial^k(m),
			\end{aligned}
		\end{equation*}
		as desired. For the latter case, let $t_0=\max\{t\in \{\max\{i,j\},\ldots,k+1\}: a_t+1\neq b_t\}$. Then $a_{t_0}+1<b_{t_0}$, and we have
		\[
		\begin{aligned}
			n&=\binom{a_{k+1}}{k+1}+\binom{a_k}{k}+\cdots+\binom{a_i}{i}\\
			&<\binom{a_{k+1}}{k+1}+\binom{a_k}{k}+\cdots+\binom{a_{t_0+1}}{t_0+1}+\binom{b_{t_0}-1}{t_0}\\
			&=\binom{b_{k+1}-1}{k+1}+\binom{b_k-1}{k}+\cdots+\binom{b_{t_0+1}-1}{t_0+1}+\binom{b_{t_0}-1}{t_0}\\
			&\leq \binom{b_{k+1}-1}{k+1}+\binom{b_k-1}{k}+\cdots+\binom{b_j-1}{j}\\
			&= \partial^k(m).
		\end{aligned}
		\]
		This proves $(a)$.
		
		Now consider $(b)$. If $n=m$, then $\partial_k(n)=\partial_k(m)$ by definition. Thus we can assume that $n<m$. Again by Lemma \ref{lem::2.2}, we have that either $i>j$ and $a_t=b_t$ for all $t\in\{i,\ldots,k+1\}$, or the maximal index $t\in\{\max\{i,j\},\ldots,k+1\}$ with $a_t\neq b_t$ satisfies $a_t<b_t$. For the former case, we have
		\[
		\begin{aligned}
			\partial_k(n)&=\binom{a_{k+1}}{k}+\binom{a_k}{k-1}+\cdots+\binom{a_i}{i-1}\\
			&= \binom{b_{k+1}}{k}+\binom{b_k}{k-1}+\cdots+\binom{b_i}{i-1}\\
			&\leq  \binom{b_{k+1}}{k}+\binom{b_k}{k-1}+\cdots+\binom{b_i}{i-1}+\cdots +\binom{b_j}{j-1}\\
			&=\partial_k(m).
		\end{aligned}
		\]
		For the latter case, let $t_0=\max\{t\in \{\max\{i,j\},\ldots,k+1\}: a_t\neq b_t\}$. Then $a_{t_0}<b_{t_0}$, and we obtain
		\begin{equation*}
			\begin{aligned}
				\partial_k(n)-\partial_k(m)&=\binom{a_{t_0}}{t_0-1}+\cdots+\binom{a_{i}}{i-1}-\binom{b_{t_0}}{t_0-1}-\cdots-\binom{b_{j}}{j-1}\\
				&\leq\binom{a_{t_0}}{t_0-1}+\cdots+\binom{a_{i}}{i-1}-\binom{b_{t_0}}{t_0-1}\\
				&\leq\binom{a_{t_0}}{t_0-1}+\cdots+\binom{a_{i}}{i-1}-\binom{a_{t_0}+1}{t_0-1}\\
				&\leq\binom{a_{t_0}}{t_0-1}+\binom{a_{t_0}-1}{t_0-2}+\cdots+\binom{a_{t_0}-t_0+1}{0}-\binom{a_{t_0}+1}{t_0-1}\\
				&=\binom{a_{t_0}+1}{t_0-1}-\binom{a_{t_0}+1}{t_0-1}\\
				&=0,
			\end{aligned}
		\end{equation*}
		where the penultimate equality follows from Pascal's identity. Thus $(b)$ follows.
		
		Next consider $(c)$. If $d^k(n)=d^k(m)$, by Lemma \ref{lem::2.2}, we assert that $i=j$ and $a_t+1=b_t+1$ for all $t\in\{i,\ldots,k+1\}$. Therefore, $n=m$, and $(c)$ follows.
	\end{proof}

	We are now ready to give the proof of Theorem \ref{thm::main1}.
	
	\renewcommand\proofname{\bf{Proof of Theorem \ref{thm::main1}}}
	\begin{proof}
		If $\partial_k(n)\leq m-\epsilon_k$, by Lemma \ref{lem::2.3}, 
		\[
		d^k(n)
		=n+\partial_k(n)
		\leq n+m-\epsilon_k.
		\]
		Since $0\leq \epsilon_k\leq \delta_k(f_k)$, by Lemma \ref{lem::2.5} $(a)$, we deduce that 
		\[
		n\leq \partial^k(n+m-\epsilon_k)=\partial^k(n+m).
		\]
		Conversely, if $ n\leq \partial^k(n+m)$, by Lemma \ref{lem::2.5} $(b)$, we have 
		\[
		\partial_k(n)\leq\partial_k(\partial^k(n+m)),
		\] 
		and hence
		\begin{equation*}
			n+\partial_k(n)\leq \partial^k(n+m)+\partial_k(\partial^k(n+m)).
		\end{equation*}
		Combining this with Lemma \ref{lem::2.4} yields that
		\begin{equation*}
			\begin{aligned}
				\partial_k(n)&\leq m-(n+m-\partial^k(n+m)-\partial_k(\partial^k(n+m)))=m-\delta_k(n+m)\leq m-\epsilon_k.
			\end{aligned}
		\end{equation*}
		This proves  $(a)$. 
		
		If $n=\partial^k(n+m)$, then $\partial_k(n)=\partial_k(\partial^k(n+m))$. This implies 
		\begin{equation}\label{eq::proof}
			n+\partial_k(n)= \partial^k(n+m)+\partial_k(\partial^k(n+m)),
		\end{equation}
		or equivalently, 
		\[
		\partial_k(n)+\delta_k(n+m)=m
		\]
		by Lemma \ref{lem::2.4}. Conversely, if $\partial_k(n)+\delta_k(n+m)=m$, then we recover \eqref{eq::proof}. By Lemma \ref{lem::2.3}, it follows that
		\[
		d^k(n)=d^k(\partial^k(n+m)).
		\]
		Combining this with  Lemma \ref{lem::2.5} $(c)$, we obtain 
		\[n=\partial^k(n+m).\] 
		Thus $(b)$ follows. 
		
		Suppose that $n=\partial^k(n+m)$. According to  $(b)$, we have $\partial_k(n)+\delta_k(n+m)=m$. If $\delta_k(n+m)=0$ then $\partial_k(n)=m$, and if $\delta_k(n+m)\neq 0$, then $\partial_k(n)\neq m$. On the other hand, if $\partial_k(n)=m$, by Lemma \ref{lem::2.3},
		\[
		d^k(n)=\partial_k(n)+n=m+n.
		\]
		Combining this with Lemma \ref{lem::2.5} $(a)$, we obtain $ n=\partial^k(n+m)$,  and so $(c)$ follows.
		
		This completes the proof.	
	\end{proof}

	 In order to prove Theorem \ref{thm::main2}, we also need the Kruskal-Katona theorem \cite{Kat66,Kru63}, which gives a complete characterization of $f$-vectors of simplicial complexes.
	
	\begin{lemma}[{Kruskal \cite{Kru63}, Katona \cite{Kat66}}] \label{lem::2.6}
		A sequence $f=(f_0,f_1,\ldots)\in \mathbb{N}_0^{(\infty)}$ is the $f$-vector of some simplicial complex if and only if 
		\[
		\partial_k(f_k)\leq f_{k-1},\quad \text{for every}~k\geq 1.
		\] 
	\end{lemma}	
	
	In what follows, we present the proof of Theorem \ref{thm::main2}.

	\renewcommand\proofname{\bf{Proof of Theorem \ref{thm::main2}}}
	\begin{proof} First observe that $(d^*)$ follows from Theorem \ref{thm::2} and Theorem \ref{thm::main1} immediately. Now suppose  that $f$ is the $f$-vector of some simplicial complex. Let $\beta\in B_f$. Since $f\sim \beta$,  by Corollary \ref{cor::1}, we have $\chi_{-1}=1$, and $f_k-\chi_{k-1}\leq \partial^k(f_k)$ for all $k\geq 1$. Also note that $f_0-\chi_{-1}=f_0-1=\partial^0(f_0)$. Hence, for each $k\geq 0$, 
		\begin{equation}\label{eq::beta_Upper}
			\begin{aligned}
				\beta_k&=f_k-\chi_{k}-\chi_{k-1}\\
				&=(f_k-\chi_{k-1})+(f_{k+1}-\chi_{k})-f_{k+1}\\
				&\leq\partial^k(f_k)+\partial^{k+1}(f_{k+1})-f_{k+1}.
			\end{aligned}
		\end{equation} 
		Furthermore, according to \cite[Theorem 5.3]{BK88}, the upper bound in  \eqref{eq::beta_Upper} can be achieved uniformly, that is,
		\[
		\psi(f)=\max B_f=(\psi_0,\psi_1,\ldots) ~\text{with}~ \psi_i=\partial^i(f_i)+\partial^{i+1}(f_{i+1})-f_{i+1}~\text{for}~i\geq 0.
		\]
		In order to prove $(c^*)$, by Theorem \ref{thm::2} and Theorem \ref{thm::main1}, it suffices to prove that  
		\[
		\beta=\psi(f)\Longleftrightarrow \chi_{-1}=1~\text{and}~ f_k-\chi_{k-1}=\partial^k(f_k)~\text{for all}~k\geq 1.
		\]
		If $\beta=\psi(f)$, then $f\sim \beta$, and hence $\chi_{-1}=1$. Moreover, from \eqref{eq::beta_Upper} we obtain $f_k-\chi_{k-1}=\partial^k(f_k)$ for all $k\geq 1$. Conversely, if $\chi_{-1}=1$ and $f_k-\chi_{k-1}=\partial^k(f_k)$ for all $k\geq 1$, then 
		\[
		\beta_k=f_k-\chi_{k}-\chi_{k-1}=f_k-\chi_{k-1}+f_{k+1}-\chi_{k}-f_{k+1}=\partial^k(f_k)+\partial^{k+1}(f_{k+1})-f_{k+1}
		\]
		for all $k\geq 0$. Therefore, $\beta=\psi(f)$. This proves $(c^*)$.
		
		Next we consider $(f^*)$. According to $(c^*)$ and $(d^*)$, if $f$ is an $f$-vector  with $\delta(f)=0$, then $\phi(\psi(f))=f$. Now suppose that $f$ is an arbitrary $f$-vector. By Lemma \ref{lem::2.6}, it is easy to see that $f-\delta(f)$ is also an $f$-vector. Moreover, by definition,  we have $\delta(f-\delta(f))=0$. Thus $\phi(\psi(f-\delta(f)))=f-\delta(f)$ by the above arguments. Note that, for any $k\geq 0$, 
		\begin{equation*}
			\begin{aligned}
				\psi_k(f_k-\delta_k(f_k))&=\partial^k(f_k-\delta_k(f_k))+\partial^{k+1}(f_{k+1}-\delta_{k+1}(f_{k+1}))-f_{k+1}+\delta_{k+1}(f_{k+1})\\
				&=\partial^k(f_k)+\partial^{k+1}(f_{k+1})-f_{k+1}+\delta_{k+1}(f_{k+1})\\
				&=\psi_k(f_k)+\delta_{k+1}(f_{k+1}).
			\end{aligned}
		\end{equation*}
		This implies $\psi(f-\delta(f))=\psi(f)+\delta_+(f)$, and consequently, 
		\[
		\phi(\psi(f)+\delta_+(f))=\phi(\psi(f-\delta(f)))=f-\delta(f).
		\]
		By \cite[Theorem 5.2]{BK88}, the mapping $\phi$ is injective and order-preserving. Therefore, we conclude that
		\begin{equation*}
			\phi(\psi(f))+\delta_+(f)\leq\phi(\psi(f)+\delta_+(f))=f-\delta(f),
		\end{equation*}
		where equalities hold everywhere if and only if $\delta(f)=0$.
		
		This completes the proof.
	\end{proof}

	\section*{Declaration of interest statement}
	
	The authors declare that they have no known competing financial interests or personal relationships that could have appeared to influence the work reported in this paper.

	\section*{Acknowledgements}
	
	X. Huang was supported by the National Natural Science Foundation of China (No. 12471324) and the Natural Science Foundation of Shanghai (No. 24ZR1415500).
	
	\section*{Data availability}
	No data was used for the research described in the article.

\end{document}